\documentclass[10pt, leqno]{article}
\usepackage{amsfonts}
\usepackage{mathrsfs}

\usepackage{euscript,amstext,amsmath,amsthm,amssymb,amscd,a4}

\newcommand{\comment}[1]{}

\newtheorem{thm}{Theorem}[section]

\newtheorem{prop}[thm]{Proposition}
 \newtheorem{cor}[thm]{Corollary}
\newtheorem{lemma}[thm]{Lemma}
\theoremstyle{remark}

\theoremstyle{definition}
\newtheorem{defn}[thm]{Definition}

\title{Burghelea-Haller analytic torsion for\\ twisted de Rham complexes}
\author{Guangxiang Su \footnote{Chern Institute of Mathematics \& LPMC, Nankai University,
Tianjin 300071, P.R. China. (guangxiangsu@googlemail.com)}}
\date{}

\begin{document}

\maketitle

\begin{abstract}
In this paper, we extend the Burghelea-Haller analytic torsion to
the twisted de Rham complexes. We also compare it with the twisted
refined analytic torsion defined by Huang.
\end{abstract}

\maketitle
\renewcommand{\theequation}{\thesection.\arabic{equation}}
\setcounter{equation}{0}

\section{Introduction}
\setcounter{equation}{0}

Let $E$ be a unitary flat vector bundle on a closed Riemannian
manifold $M$. In \cite{RS}, Ray and Singer defined an analytic
torsion associated to $(M,E)$ and proved that it does not depend on
the Riemannian metric on $M$. Moreover, they conjectured that this
analytic torsion coincides with the classical Reidemeister torsion
defined using a triangulation on $M$ (cf. \cite{Mi}). This
conjecture was later proved in the celebrated papers of Cheeger
\cite{C} and M\"{u}ller \cite{Mu1}. M\"{u}ller generalized this
result in \cite{Mu2} to the case when $E$ is a unimodular flat
vector bundle on $M$. In \cite{BZ1}, inspired by the considerations
of Quillen \cite{Q}, Bismut and Zhang reformulated the above
Cheeger-M\"{u}ller theorem as an equality between the Reidemeister
and Ray-Singer metrics defined on the determinant of cohomology, and
proved an extension of it to the case of general flat vector bundle
over $M$. The method used in \cite{BZ1} is different from those of
Cheeger and M\"{u}ller in that it makes use of a deformation by
Morse functions introduced by Witten \cite{W} on the de Rham
complex.

Braverman and Kappeler \cite{BK1,BK2,BK3} defined the refined
analytic torsion for flat vector bundle over odd dimensional
manifolds, and show that it equals to the Turaev torsion (cf.
\cite{FT,T}) up to a multiplication by a complex number of absolute
value one. Burghelea and Haller \cite{BH1,BH2}, following a
suggestion of M\"{u}ller, defined a generalized analytic torsion
associated to a non-degenerate symmetric bilinear form on a flat
vector bundle over an arbitrary dimensional manifold and make an
explicit conjecture between this generalized analytic torsion and
the Turaev torsion. This conjecture was proved up to sign by
Burghelea-Haller \cite{BH3} and in full generality by Su-Zhang
\cite{SZ}. Cappell and Miller \cite{CM} used non-self-adjoint
Laplace operators to define another complex valued analytic torsion
and used the method in \cite{SZ} to prove an extension of the
Cheeger-M\"{u}ller theorem.

In \cite{MW,MW1}, Mathai and Wu generalized the classical Ray-Singer
analytic torsion to the twisted de Rham complex with an odd degree
closed differential form $H$. Recently, Huang \cite{H} generalized
the refined analytic torsion \cite{BK1,BK2,BK3} to the twisted de
Rham complex, got a duality theorem and compared it with the twisted
Ray-Singer metric which also was defined in \cite{H}.

In this paper, suppose there exists a non-degenerate symmetric
bilinear form on the flat vector bundle $E$, we generalize the
Burghelea-Haller analytic torsion to the twisted de Rham complex.
For the odd dimensional manifold, we also compare it with the
twisted refined analytic torsion and the twisted Ray-Singer metric.

The rest of this paper is organized as follows. In Section 2,
suppose there exists a $\mathbb{Z}_{2}$-graded non-degenerate
symmetric bilinear form on a $\mathbb{Z}_{2}$-graded finite
dimensional complex, we define a symmetric bilinear torsion on it.
In Section 3, we generalize the Burghelea-Haller analytic torsion to
the twisted de Rham complex. In Section 4, when the dimension of the
manifold is odd, we show that the twisted Burghelea-Haller analytic
torsion is independent of the Riemannian metric $g$, the symmetric
bilinear form $b$ and the representative $H$ in the cohomology class
$[H]$. In Section 5, we compare it with the twisted refined analytic
torsion. In Section 6, we briefly discuss the Cappell-Miller
analytic torsion on the twisted de Rham complex.

\section{Symmetric bilinear torsion on a finite dimensional $\mathbb{Z}_{2}$-graded complex}
\setcounter{equation}{0} Let
$$\begin{CD}0\longrightarrow C^{0}@>d_{0}>>C^{1}@>d_{1}>>\cdots@>d_{n-1}>>C^{n}\longrightarrow 0
\end{CD}$$
be a cochain complex of finite dimensional complex vector space. Set
$$C^{\bar{k}}=\bigoplus_{i=k\ {\rm mod}\ 2}C^{i},\ d_{\bar{k}}=\sum_{i=k\ {\rm mod}\ 2}d_{i}\ \ k=0,1.$$
Then we get a $\mathbb{Z}_{2}$-graded cochain complex
\begin{align}
\begin{CD}
\cdots@>d_{\bar{1}}>>C^{\bar{0}}@>d_{\bar{0}}>>C^{\bar{1}}@>d_{\bar{1}}>>C^{\bar{0}}@>d_{\bar{0}}>>\cdots.
\end{CD}
\end{align}
Denote its cohomology by $H^{\bar{k}},\ k=0,1$. Set
$${\rm det}(C^{\bullet},d)={\rm det}C^{\bar{0}}\otimes \left({\rm det}C^{\bar{1}}\right)^{-1},\ \
{\rm det}(H^{\bullet},d)={\rm det}H^{\bar{0}}\otimes\left({\rm
det}H^{\bar{1}}\right)^{-1}.$$ Then we have a canonical isomorphism
between the determinant lines
\begin{align}\label{2.2}
\phi:{\rm det}(C^{\bullet},d)\longrightarrow{\rm
det}(H^{\bullet},d).
\end{align}

Suppose that there is a non-degenerate symmetric bilinear form on
$C^{\bar{k}}$, $k=0,1$. Then it induces a non-degenerate symmetric
bilinear form $b_{{\rm det}H^{\bullet}(C^{\bullet},d)}$ on the
determinant line ${\rm det}(H^{\bullet},d)$ via the isomorphism
(\ref{2.2}).
 Let
$d_{\bar{k}}^{\#}$ be the adjoint of $d_{\bar{k}}$ with respect to
the non-degenerate symmetric bilinear form and define
$\Delta_{b,\bar{k}}=d_{\bar{k}}^{\#}d_{\bar{k}}+d_{\overline{k+1}}d_{\overline{k+1}}^{\#}$.
Let us write $C_{b}^{\bar{k}}(\lambda)$ for the generalized
$\lambda$-eigen space of $\Delta_{b,\bar{k}}$. Then we have a
$b$-orthogonal decomposition
\begin{align}\label{2.3}
C^{\bar{k}}=\bigoplus_{\lambda}C^{\bar{k}}_{b}(\lambda)
\end{align}
and the inclusion $C^{\bar{k}}_{b}(0)\to C^{\bar{k}}$ induces an
isomorphism in cohomology. Particularly, we obtain a canonical
isomorphism
\begin{align}
{\rm det}H^{\bullet}(C^{\bullet}_{b}(0))\cong {\rm
det}H^{\bullet}(C^{\bullet}).
\end{align}
\begin{prop}\label{t2.1}
The following identity holds,
\begin{multline}
b_{{\rm det}H^{\bullet}(C^{\bullet},d)}\\=b_{{\rm
det}H^{\bullet}(C^{\bullet}_{b}(0),d)}\cdot{\rm
det}\left(\left.d_{\bar{0}}^{\#}d_{\bar{0}}\right|_{{C_{b}^{\bar{0},\perp}(0)}\cap{\rm
im}d_{\bar{0}}^{\#}}\right)^{-1}\cdot{\rm
det}\left(\left.d_{\bar{1}}^{\#}d_{\bar{1}}\right|_{{C_{b}^{\bar{1},\perp}(0)}\cap{\rm
im}d_{\bar{1}}^{\#}}\right),
\end{multline}
where $C^{\bar{k},\perp}_{b}(0)=\oplus_{\lambda\neq
0}C^{\bar{k}}_{b}(\lambda)$, $k=0,1$.
\end{prop}
\begin{proof}
The proof is the same as \cite[Lemma 3.3]{BH2}. Suppose
$(C_{1}^{\bullet},b_{1})$ and $(C_{2}^{\bullet},b_{2})$ are
finite-dimensional $\mathbb{Z}_{2}$-graded complexes equipped with
$\mathbb{Z}_{2}$-graded non-degenerate symmetric bilinear forms.
Clearly, $H^{\bullet}(C_{1}^{\bullet}\oplus
C_{2}^{\bullet})=H^{\bullet}(C_{1}^{\bullet})\oplus
H^{\bullet}(C_{2}^{\bullet})$ and we obtain a canonical isomorphism
of determinant lines
$${\rm det}H^{\bullet}(C_{1}^{\bullet}\oplus C_{2}^{\bullet})={\rm det}H^{\bullet}(C_{1}^{\bullet})
\otimes {\rm det}H^{\bullet}(C_{2}^{\bullet}).$$ Then we have
$$b_{{\rm det}H^{\bullet}(C_{1}^{\bullet}\oplus C_{2}^{\bullet})}=b_{{\rm det}H^{\bullet}(C_{1}^{\bullet})}
\otimes b_{{\rm det}H^{\bullet}(C_{2}^{\bullet})}.$$ In view of the
$b$-orthogonal decomposition (\ref{2.3}) we may therefore without
loss of generality assume ${\rm ker}\Delta_{b,\bar{k}}=0,\ k=0,1$.
Then by \cite[Lemma 3.3]{BH2} we have
$$C^{\bar{k}}={\rm im}d_{\overline{k+1}}\oplus {\rm im}d_{\bar{k}}^{\#}.$$
This decomposition is $b$-orthogonal and invariant under
$\Delta_{b}$. Then we have the following two exact complexes
$$
\begin{CD}
0\longrightarrow C^{\bar{0}}\cap{\rm
im}d_{\bar{0}}^{\#}@>d_{\bar{0}}>>C^{\bar{1}}\cap{\rm
im}d_{\bar{0}}\longrightarrow0 \end{CD}$$ and
$$
\begin{CD}
0\longrightarrow C^{\bar{1}}\cap{\rm
im}d_{\bar{1}}^{\#}@>d_{\bar{1}}>>C^{\bar{0}}\cap{\rm
im}d_{\bar{1}}\longrightarrow0 \end{CD}.$$ Then from \cite[Example
3.2]{BH2}, we get the proposition.
\end{proof}

\section{Symmetric bilinear torsion on the twisted de Rham complexes}
\setcounter{equation}{0}

In this section, we suppose that there is a fiber-wise
non-degenerate symmetric bilinear form on $E$. Then we define a
symmetric bilinear torsion on the determinant line of the twisted de
Rham complex.

\subsection{Twisted de Rham complexes}
In this subsection, we review the twisted de Rham complexes from
\cite{MW}.

Let $M$ be a closed Riemannian manifold and $E\to M$ be a complex
flat vector bundle with flat connection $\nabla$. Let $H$ be an
odd-degree closed differential form on $M$. We set
$\Omega^{\bar{0}}=\Omega^{\rm even}(M,E)$,
$\Omega^{\bar{1}}=\Omega^{\rm odd}(M,E)$ and
$\nabla^{H}=\nabla+H\wedge.$ We define the twisted de Rham
cohomology groups as
$$H^{\bar{k}}(M,E,H)={{{\rm ker}\left(\nabla^{H}:\Omega^{\bar{k}}(M,E)\to \Omega^{\overline{k+1}}(M,E)\right)}
\over{{\rm im}\left(\nabla^{H}:\Omega^{\overline{k+1}}(M,E)\to
\Omega^{\bar{k}}(M,E)\right)}},\ \ k=0,1.$$ Suppose $H$ is replaced
by $H'=H-dB$ for some $B\in\Omega^{\bar{0}}(M)$, then there is an
isomorphism
$\varepsilon_{B}=e^{B}\wedge\cdot:\Omega^{\bullet}(M,E)\to
\Omega^{\bullet}(M,E)$ satisfying
$$\varepsilon_{B}\circ\nabla^{H}=\nabla^{H'}\circ\varepsilon_{B}.$$
Therefore $\varepsilon_{B}$ induces an isomorphism
$$\varepsilon_{B}:H^{\bullet}(M,E,H)\to H^{\bullet}(M,E,H')$$
on the twisted de Rham cohomology.

\subsection{The construction of the symmetric bilinear torsion}
Suppose that there exists a non-degenerate symmetric bilinear form
on $E$. To simplify notation, let
$C^{\bar{k}}=\Omega^{\bar{k}}(X,E)$ and let
$d_{\bar{k}}=d_{\bar{k}}^{E,H}$ be the operator $\nabla^{H}$ acting
on $C^{\bar{k}}$ ($k=0,1$). Then
$d_{\bar{1}}d_{\bar{0}}=d_{\bar{0}}d_{\bar{1}}=0$ and we have a
complex
\begin{align}
\begin{CD}
\cdots@>d_{\bar{1}}>>C^{\bar{0}}@>d_{\bar{0}}>>C^{\bar{1}}@>d_{\bar{1}}>>C^{\bar{0}}@>d_{\bar{0}}>>\cdots.
\end{CD}
\end{align}
The metric $g^{M}$ and the symmetric bilinear form $b$ determine
together a symmetric bilinear form on $\Omega^{\bullet}(M,E)$ such
that if $u=\alpha f$, $v=\beta g\in\Omega^{\bullet}(M,E)$ such that
$\alpha,\beta\in \Omega^{\bullet}(M)$, $f,g\in\Gamma(E)$, then
\begin{align}\label{1}
\beta_{g,b}(u,v)=\int_{M}(\alpha\wedge*\beta)b(f,g),
\end{align}
where $*$ is the Hodge star operator. Denote by $d_{\bar{k}}^{\#}$
the adjoint of $d_{\bar{k}}$ with respect to the non-degenerate
symmetric bilinear form (\ref{1}). Then the Laplacians
$$\Delta_{b,\bar{k}}=d_{\bar{k}}^{\#}d_{\bar{k}}+d_{\overline{k+1}}d_{\overline{k+1}}^{\#},\ \ k=0,1$$
If $\lambda$ is in the spectrum of $\Delta_{b,\bar{k}}$, then the
image of the associated spectral projection is finite dimensional
and contains smooth forms only. We refer to this image as the
(generalized) $\lambda$-eigen space of $\Delta_{b,\bar{k}}$ and
denote it by $\Omega^{\bar{k}}_{\{\lambda\}}(M,E)$ and there exists
$N_{\lambda}\in\mathbb{N}$ such that
$$\left(\Delta_{b,\bar{k}}-\lambda\right)^{N_{\lambda}}|_{\Omega^{\bar{k}}_{\{\lambda\}}(M,E)}=0.$$
Then for different generalized eigenvalues $\lambda,\mu$, the spaces
$\Omega^{\bar{k}}_{\{\lambda\}}(M,E)$ and
$\Omega^{\bar{k}}_{\{\mu\}}(M,E)$ are $\beta_{g,b}$-orthogonal.

For any $a\geq 0$, set
$$\Omega^{\bar{k}}_{[0,a]}(M,E)=\bigoplus_{0\leq |\lambda|\leq a}\Omega^{\bar{k}}_{\{\lambda\}}(M,E).$$
Then $\Omega^{\bar{k}}_{[0,a]}(M,E)$ is of finite dimensional and
one gets a non-degenerate symmetric bilinear form $b_{{\rm
det}H^{\bullet}(\Omega^{\bullet}_{[0,a]},d)}$ on ${\rm
det}H^{\bullet}(\Omega^{\bullet}_{[0,a]},d)$. Let
$\Omega^{\bar{k}}_{(a,+\infty)}(M,E)$ denote the
$\beta_{g,b}$-orthogonal complement to
$\Omega^{\bar{k}}_{[0,a]}(M,E)$.

For the subcomplexes
$(\Omega^{\overline{k+1}}_{(a,+\infty)}(M,E),d)$, since the
operators $d_{\bar{k}}d_{\bar{k}}^{\#}$ and
$\Delta_{b,\overline{k+1}}$ are equal and invertible on ${\rm
im}(d_{\bar{k}})\cap\Omega^{\overline{k+1}}_{(a,+\infty)}(M,E)$, we
have
\begin{align}
P_{\bar{k}}:=d_{\bar{k}}^{\#}\left(d_{\bar{k}}d_{\bar{k}}^{\#}\right)^{-1}d_{\bar{k}}
=d_{\bar{k}}^{\#}\left(\Delta_{b,\overline{k+1}}\right)^{-1}d_{\bar{k}}
\end{align}
is a pseudodifferntial operator of order $0$ and satisfies
$$P_{\bar{k}}^{2}=P_{\bar{k}}.$$
 By definition we have
\begin{multline}
\zeta\left(s,d_{\bar{k}}^{\#}d_{\bar{k}}|_{{\rm
im}d_{\bar{k}}^{\#}\cap\Omega^{\bar{k}}_{(a,+\infty)}(M,E)}\right)={\rm
Tr}\left(\Delta_{b,\bar{k}}^{-s}P_{\bar{k}}|_{\Omega^{\bar{k}}_{(a,+\infty)}(M,E)}\right)\\
={\rm
Tr}\left(P_{\bar{k}}\Delta_{b,\bar{k}}^{-s}|_{\Omega^{\bar{k}}_{(a,+\infty)}(M,E)}\right).
\end{multline}
Then $\zeta\left(s,d_{\bar{k}}^{\#}d_{\bar{k}}|_{{\rm
im}d_{\bar{k}}^{\#}\cap\Omega^{\bar{k}}_{(a,+\infty)}(M,E)}\right)$
has a meromorphic extension to the whole complex plane and, by
\cite[Section 7]{Wo}, it is regular at $0$. Then by \cite{GS,Wo}, we
have the following result which is an analogue of \cite[Theorem
2.1]{MW}.
\begin{thm}
For $k=0,1$, $\zeta\left(s,d_{\bar{k}}^{\#}d_{\bar{k}}|_{{\rm
im}d_{\bar{k}}^{\#}\cap\Omega^{\bullet}_{(a,+\infty)}(M,E)}\right)$
is holomorphic in the half plane for ${\rm Re}(s)>n/2$ and extends
meromorphically to $\mathbb{C}$ with possible poles at
$\{{{n-l}\over 2},l=0,1,2,\dots\}$ only, and is holomorphic at
$s=0$.
\end{thm}
Then for $k=0,1$ and any $a\geq 0$, the following regularized zeta
determinant is well defined:
\begin{align}\label{3.4}
{\rm
det}'\left(d_{\bar{k}}^{\#}d_{\bar{k}}|_{\Omega^{\bar{k}}_{(a,+\infty)}(M,E)}\right):=
\exp\left(-\zeta'\left(0,d_{\bar{k}}^{\#}d_{\bar{k}}|_{{\rm
im}d_{\bar{k}}^{\#}\cap\Omega^{\bar{k}}_{(a,+\infty)}(M,E)}\right)\right).
\end{align}
\begin{prop}
The symmetric bilinear form on ${\rm
det}H^{\bullet}(\Omega^{\bullet}(M,E,H),d)$ defined by
\begin{align}\label{3.3}
b_{{\rm det}H^{\bullet}(\Omega^{\bullet}_{[0,a]}(M,E),d)}\cdot{\rm
det}'\left(d_{\bar{0}}^{\#}d_{\bar{0}}|_{\Omega^{\bar{0}}_{(a,+\infty)}(M,E)}\right)^{-1}\cdot\left({\rm
det}'\left(d_{\bar{1}}^{\#}d_{\bar{1}}|_{\Omega^{\bar{1}}_{(a,+\infty)}(M,E)}\right)\right)
\end{align}
is independent of the choice of $a\geq 0$.
\end{prop}
\begin{proof}
Let $0\leq a<c<\infty$. We have
\begin{align}\label{3.5}
\left(\Omega^{\bar{k}}_{[0,c]}(M,E),d_{\bar{k}}\right)=\left(\Omega^{\bar{k}}_{[0,a]}(M,E),d_{\bar{k}}\right)
\bigoplus \left(\Omega^{\bar{k}}_{(a,c]}(M,E),d_{\bar{k}}\right)
\end{align}
and
$$\left(\Omega^{\bar{k}}_{(a,+\infty)}(M,E),d_{\bar{k}}\right)=\left(\Omega^{\bar{k}}_{(a,c]}(M,E),d_{\bar{k}}\right)
\bigoplus
\left(\Omega^{\bar{k}}_{(c,+\infty)}(M,E),d_{\bar{k}}\right).$$ Then
by definition of the determinant, we get
\begin{multline}
{\rm
det}'\left(d_{\bar{k}}^{\#}d_{\bar{k}}|_{\Omega^{\bar{1}}_{(a,+\infty)}(M,E)}\right)\\
={\rm
det}'\left(d_{\bar{k}}^{\#}d_{\bar{k}}|_{\Omega^{\bar{1}}_{(a,c]}(M,E)}\right)\cdot{\rm
det}'\left(d_{\bar{k}}^{\#}d_{\bar{k}}|_{\Omega^{\bar{1}}_{(c,+\infty)}(M,E)}\right).
\end{multline}
Applying Proposition \ref{t2.1} to (\ref{3.5}), we get
$$b_{{\rm det}H^{\bullet}(\Omega^{\bullet}_{[0,c]})}=
b_{{\rm det}H^{\bullet}(\Omega^{\bullet}_{[0,a]})}\cdot{\rm
det}'\left(d_{\bar{0}}^{\#}d_{\bar{0}}|_{\Omega^{\bar{0}}_{(a,c]}(M,E)}\right)^{-1}\cdot\left({\rm
det}'\left(d_{\bar{1}}^{\#}d_{\bar{1}}|_{\Omega^{\bar{1}}_{(a,c]}(M,E)}\right)\right).$$
Then we get the proposition.
\end{proof}
\begin{defn}
The symmetric bilinear form defined by (\ref{3.3}) is called the
Ray-Singer symmetric bilinear torsion on ${\rm
det}H^{\bullet}(\Omega^{\bullet}(M,E,H),d)$ and is denoted by
$\tau_{b,\nabla,H}$.
\end{defn}

\section{Symmetric bilinear torsion under metric and flux deformations}
\setcounter{equation}{0} In this section, we will use the methods in
\cite{MW} to study the dependence of the torsion on the metric $g$,
the symmetric bilinear form $b$ and the flux $H$.

\subsection{Variation of the torsion with respect to the metric and symmetric bilinear form}
We assume that $M$ is a closed compact oriented manifold of odd
dimension. Suppose the pair $(g_{u},b_{u})$ is deformed smoothly
along a one-parameter family with parameter $u\in\mathbb{R}$. Let
$Q_{\bar{k}}$ be the spectral projection onto
$\Omega^{\bar{k}}_{[0,a]}(M,E)$ and $\Pi_{\bar{k}}=1-Q_{\bar{k}}$ be
the spectral projection onto $\Omega^{\bar{k}}_{(a,+\infty)}(M,E)$.
Let
$$\alpha=*_{u}^{-1}{{\partial*_{u}}\over{\partial u}}+b_{u}^{-1}{{\partial b_{u}}\over{\partial u}}.$$
\begin{lemma}\label{t4.1}
Under the above assumptions,
\begin{multline}
{\partial\over{\partial u}}\log\left[{\rm
det}'\left(d_{\bar{0}}^{\#}d_{\bar{0}}|_{\Omega^{\bar{0}}_{(a,+\infty)}(M,E)}\right)^{-1}\cdot\left({\rm
det}'\left(d_{\bar{1}}^{\#}d_{\bar{1}}|_{\Omega^{\bar{1}}_{(a,+\infty)}(M,E)}\right)\right)\right]\\
=-\sum_{k=0,1}(-1)^{k}{\rm Tr}(\alpha Q_{\bar{k}}).\\
\end{multline}
\end{lemma}
\begin{proof}
While $d_{\bar{k}}$ is independent of $u$, we have
$${{\partial d_{\bar{k}}^{\#}}\over{\partial u}}=-\left[\alpha,d_{\bar{k}}^{\#}\right].$$
Using $P_{\bar{k}}d_{\bar{k}}^{\#}=d_{\bar{k}}^{\#}$,
$d_{\bar{k}}P_{\bar{k}}=d_{\bar{k}}$ and
$P_{\bar{k}}^{2}=P_{\bar{k}}$, we get
$d_{\bar{k}}^{\#}d_{\bar{k}}P_{\bar{k}}=P_{\bar{k}}d_{\bar{k}}^{\#}d_{\bar{k}}=d_{\bar{k}}^{\#}d_{\bar{k}}$
and
$${{\partial P_{\bar{k}}}\over{\partial u}}={{\partial P_{\bar{k}}}\over{\partial u}}P_{\bar{k}},\ \
P_{\bar{k}}{{\partial P_{\bar{k}}}\over{\partial u}}=0.$$ Following
the $\mathbb{Z}$-graded case, we set
\begin{multline}
f(s,u)=\sum_{k=0,1}(-1)^{k}\int_{0}^{+\infty}t^{s-1}{\rm
Tr}\left(e^{-td_{\bar{k}}^{\#}d_{\bar{k}}}P_{\bar{k}}|_{\Omega^{\bar{k}}_{(a,+\infty)}(M,E)}\right)dt\\
=\Gamma(s)\sum_{k=0,1}(-1)^{k}\zeta\left(s,d_{\bar{k}}^{\#}d_{\bar{k}}|_{\Omega^{\bar{k}}_{(a,+\infty)}(M,E)}\right).\\
\end{multline}
Using the above identities on $P_{\bar{k}}$, the trace property and
by an application of Duhamel's principal, we get
\begin{multline}
{{\partial f}\over{\partial
u}}\\=\sum_{k=0,1}(-1)^{k}\int_{0}^{+\infty}t^{s-1}{\rm
Tr}\left(t\left[\alpha,d_{\bar{k}}^{\#}\right]d_{\bar{k}}e^{-td_{\bar{k}}^{\#}d_{\bar{k}}}\Pi_{\bar{k}}+
e^{-td_{\bar{k}}^{\#}d_{\bar{k}}}{{\partial
P_{\bar{k}}}\over{\partial u}}P_{\bar{k}}\Pi_{\bar{k}}\right)dt\\
=\sum_{k=0,1}(-1)^{k}\int_{0}^{+\infty}t^{s-1}{\rm
Tr}\left(t\alpha\left[d_{\bar{k}}^{\#},d_{\bar{k}}e^{-td_{\bar{k}}^{\#}d_{\bar{k}}}\right]\Pi_{\bar{k}}
+P_{\bar{k}}e^{-td_{\bar{k}}^{\#}d_{\bar{k}}}{{\partial P_{\bar{k}}}\over{\partial u}}\Pi_{\bar{k}}\right)dt\\
=\sum_{k=0,1}(-1)^{k}\int_{0}^{+\infty}t^{s-1}{\rm
Tr}\left(t\alpha\left(e^{-td_{\bar{k}}^{\#}d_{\bar{k}}}d_{\bar{k}}^{\#}d_{\bar{k}}-
e^{-td_{\bar{k}}d_{\bar{k}}^{\#}}d_{\bar{k}}d_{\bar{k}}^{\#}\right)\Pi_{\bar{k}}
+e^{-td_{\bar{k}}^{\#}d_{\bar{k}}}P_{\bar{k}}{{\partial
P_{\bar{k}}}\over{\partial
u}}\Pi_{\bar{k}}\right)dt\\=\sum_{k=0,1}(-1)^{k}\int_{0}^{+\infty}t^{s}{\rm
Tr}\left(\alpha e^{-t\Delta_{b,\bar{k}}}
\Delta_{b,\bar{k}}\Pi_{\bar{k}}\right)dt\\
=-\sum_{k=0,1}(-1)^{k}\int_{0}^{+\infty}t^{s}{\partial\over{\partial
t}}{\rm
Tr}\left(\alpha\left(e^{-t\Delta_{b,\bar{k}}}\Pi_{\bar{k}}\right)\right)dt.
\end{multline}
Integrating by parts, we have
\begin{multline}
{{\partial f}\over{\partial
u}}=s\sum_{k=0,1}(-1)^{k}\int_{0}^{+\infty}t^{s-1}{\rm
Tr}\left(\alpha\left(e^{-t\Delta_{b,\bar{k}}}\Pi_{\bar{k}}\right)\right)dt\\
=s\sum_{k=0,1}(-1)^{k}\left(\int_{0}^{1}+\int_{1}^{+\infty}\right)t^{s-1}{\rm
Tr}\left(\alpha e^{-t\Delta_{b,\bar{k}}}(1-Q_{\bar{k}})\right)dt.
\end{multline}
Since $\alpha$ is a smooth tensor and $n$ is odd, the asymptotic
expansion as $t\downarrow 0$ for ${\rm Tr}(\alpha
e^{-t\Delta_{b,\bar{k}}})$ does not contain a constant term.
Therefore
$$\int_{0}^{1}t^{s-1}{\rm Tr}\left(\alpha e^{-t\Delta_{b,\bar{k}}}\right)dt$$
does not have a pole at $s=0$. On the other hand, because of the
exponential decay of ${\rm Tr}(\alpha
e^{-t\Delta_{b,\bar{k}}}\Pi_{\bar{k}})$ for large $t$,
$$\int_{1}^{+\infty}t^{s-1}{\rm Tr}\left(\alpha e^{-t\Delta_{b,\bar{k}}}\Pi_{\bar{k}}\right)$$
is an entire function in $s$. So
\begin{align}\label{4.5}
\left.{{\partial f}\over{\partial
u}}\right|_{s=0}=-s\sum_{k=0,1}(-1)^{k}\int_{0}^{1}\left.t^{s-1}{\rm
Tr}(\alpha Q_{\bar{k}})dt\right|_{s=0} =-\sum_{k=0,1}(-1)^{k}{\rm
Tr}\left(\alpha Q_{\bar{k}}\right)
\end{align}
and hence
\begin{align}\label{4.6}
{\partial\over{\partial
u}}\sum_{k=0,1}(-1)^{k}\zeta\left(0,d_{\bar{k}}^{\#}d_{\bar{k}}|_{\Omega^{\bar{k}}_{(a,+\infty)}(M,E)}\right)=0.
\end{align}
Finally, the result follows from (\ref{4.5}), (\ref{4.6}) and
\begin{multline}
\log\left[{\rm
det}'\left(d_{\bar{0}}^{\#}d_{\bar{0}}|_{\Omega^{\bar{0}}_{(a,+\infty)}(M,E)}\right)^{-1}\cdot\left({\rm
det}'\left(d_{\bar{1}}^{\#}d_{\bar{1}}|_{\Omega^{\bar{1}}_{(a,+\infty)}(M,E)}\right)\right)\right]\\
=\lim_{s\to 0}\left[f(s,u)-{1\over
s}\sum_{k=0,1}(-1)^{k}\zeta\left(0,d_{\bar{k}}^{\#}d_{\bar{k}}|_{\Omega^{\bar{k}}_{(a,+\infty)}(M,E)}\right)\right].
\end{multline}
\end{proof}

\begin{lemma}\label{t4.2}
Under the same assumptions, along any one-parameter deformation of
$(g_{u},b_{u})$, we have
\begin{align}\label{4.8}
\left.{{\partial}\over{\partial w}}\right|_{u}\left({{b_{w,{\rm
det}H^{\bullet}(\Omega^{\bullet}_{[0,a]}(M,E),d)}} \over{b_{u,{\rm
det}H^{\bullet}(\Omega^{\bullet}_{[0,a]}(M,E),d)}}}\right)=\sum_{k=0,1}(-1)^{k}{\rm
Tr}\left(\alpha Q_{\bar{k}}\right).
\end{align}
\end{lemma}
\begin{proof}
For sufficiently small $w-u$, the restriction of the spectral
projection
$$Q_{\bar{k}}|_{\Omega^{\bar{k}}_{u,[0,a]}(M,E)}:\Omega^{\bar{k}}_{u,[0,a]}(M,E)\longrightarrow \Omega^{\bar{k}}_{w,[0,a]}(M,E)$$
is an isomorphism of complexes. Then for sufficiently small $w-u$,
we have
\begin{multline}
{{b_{w,{\rm det}H^{\bullet}(\Omega^{\bullet}_{[0,a]}(M,E),d)}}
\over{b_{u,{\rm
det}H^{\bullet}(\Omega^{\bullet}_{[0,a]}(M,E),d)}}}={\rm
det}\left(\left(\beta_{g_u,b_u}|_{\Omega^{\bar{0}}_{u,[0,a]}(M,E)}\right)^{-1}
\left(Q_{\bar{0}}|_{\Omega^{\bar{0}}_{u,[0,a]}(M,E)}\right)^{*}\right.\\
\left.\left(\beta_{g_w,b_w}|_{\Omega^{\bar{0}}_{w,[0,a]}(M,E)}\right)\right)\cdot
{\rm
det}\left(\left(\beta_{g_u,b_u}|_{\Omega^{\bar{1}}_{u,[0,a]}(M,E)}\right)^{-1}
\left(Q_{\bar{1}}|_{\Omega^{\bar{1}}_{u,[0,a]}(M,E)}\right)^{*}\right.\\
\left.\left(\beta_{g_w,b_w}|_{\Omega^{\bar{1}}_{w,[0,a]}(M,E)}\right)\right)^{-1}.
\end{multline}
Then similarly as in \cite{BH2}, we get (\ref{4.8}).

\end{proof}

Combining Lemma \ref{t4.1} and Lemma \ref{t4.2}, we have
\begin{thm}
Let $M$ be a closed, compact manifold of odd dimension, $E$ be a
flat vector bundle over $M$, and $H$ be a closed differential form
on $M$ of odd degree. Then the symmetric bilinear torsion
$\tau_{b,\nabla,H}$ on the twisted de Rham complex does not depend
on the choices of the Riemannian metric on $M$ and the symmetric
bilinear form $b$ in a same homotopy class of non-degenerate
symmetric bilinear forms on $E$.
\end{thm}

\subsection{Variation of analytic torsion with respect to the flux in a cohomology class}
We continue to assume that ${\rm dim}M$ is odd and use the same
notation as above. Suppose the (real) flux form $H$ is deformed
smoothly along a one-parameter family with parameter
$v\in\mathbb{R}$ in such a way that the cohomology class $[H]\in
H^{\bar{1}}(M,\mathbb{R})$ is fixed. Then ${{\partial
H}\over{\partial v}}=-dB$ for some form $B\in \Omega^{\bar{0}}(M)$
that depends smoothly on $v$; let
$$\beta=B\wedge\cdot.$$
\begin{lemma}\label{t4.4}
Under the above assumptions,
\begin{multline}
{\partial\over{\partial v}}\log\left[{\rm
det}'\left(d_{\bar{0}}^{\#}d_{\bar{0}}|_{\Omega^{\bar{0}}_{(a,+\infty)}(M,E)}\right)^{-1}\cdot\left({\rm
det}'\left(d_{\bar{1}}^{\#}d_{\bar{1}}|_{\Omega^{\bar{1}}_{(a,+\infty)}(M,E)}\right)\right)\right]\\
=-2\sum_{k=0,1}(-1)^{k}{\rm Tr}(\beta Q_{\bar{k}}).
\end{multline}
\end{lemma}
\begin{proof}
As in the proof of Lemma \ref{t4.1}, we set
$$f(s,v)=\sum_{k=0,1}(-1)^{k}\int_{0}^{+\infty}t^{s-1}{\rm Tr}\left(e^{-td_{\bar{k}}^{\#}d_{\bar{k}}}P_{\bar{k}}
|_{\Omega^{\bar{k}}_{(a,+\infty)}(M,E)}\right)dt.$$ We note that
$B$, hence $\beta$ is real. Using
$${{\partial d_{\bar{k}}}\over{\partial v}}=\left[\beta,d_{\bar{k}}\right],\ {{\partial d_{\bar{k}}^{\#}}
\over{\partial v}}=-\left[\beta^{\#},d_{\bar{k}}^{\#}\right],$$
$$P_{\bar{k}}^{2}=P_{\bar{k}}=P_{\bar{k}}^{\#},\ P_{\bar{k}}{{\partial P_{\bar{k}}}\over{\partial v}}P_{\bar{k}}=0$$
and by Dumahel's principle, similarly as \cite[Lemma 3.5]{MW}, we
get
\begin{align}
{{\partial f}\over{\partial
v}}=-2\sum_{k=0,1}(-1)^{k}\int_{0}^{+\infty}t^{s}{\partial\over{\partial
t}}{\rm Tr}\left(\beta
e^{-t\Delta_{b,\bar{k}}}\Pi_{\bar{k}}\right)dt.
\end{align}
The rest is similar to the proof of Lemma \ref{t4.1}.
\end{proof}

\begin{lemma}\label{t4.5}
Under the same assumptions, along any one-parameter deformation of
$H$ that fixes the cohomology class $[H]$, then we have
\begin{align}\label{4.11}
\left.{{\partial}\over{\partial w}}\right|_{v}\left({{b_{{\rm
det}H^{\bullet}(\Omega^{\bullet}_{[0,a]}(M,E,H^{w}),d)}}
\over{b_{{\rm
det}H^{\bullet}(\Omega^{\bullet}_{[0,a]}(M,E,H^{v}),d)}}}\right)=2\sum_{k=0,1}(-1)^{k}{\rm
Tr}\left(\beta Q_{\bar{k}}\right),
\end{align}
where we identify ${\rm det}H^{\bullet}(M,E,H)$ along the
deformation.
\end{lemma}
\begin{proof}
For sufficiently small $w-v$, we have
$$Q_{\bar{k}}\varepsilon_{B}:\Omega^{\bar{k}}_{[0,a]}(M,E,H^{v})\longrightarrow
\Omega^{\bar{k}}_{[0,a]}(M,E,H^{w})$$ is an isomorphism of complexes
and the induced symmetric bilinear form on the determinant line
${\rm det}H^{\bullet}(\Omega^{\bullet}_{[0,a]}(M,E,H^{u}),d)$ is
\begin{multline}
\left(\left({\rm
det}\left(Q_{\bar{k}}\varepsilon_{B}\right)^*{b_{{\rm
det}H^{\bullet}(\Omega^{\bullet}_{[0,a]}(M,E,H^{w}),d)}}\right)\right)(\cdot\ ,\ \cdot)\\
={b_{{\rm
det}H^{\bullet}(\Omega^{\bullet}_{[0,a]}(M,E,H^{w}),d)}}\left({\rm
det}\left(Q_{\bar{k}}\varepsilon_{B}\right)\cdot\ ,\ {\rm
det}\left(Q_{\bar{k}}\varepsilon_{B}\right)\cdot\right),
\end{multline}
where
$${\rm det}\left(Q_{\bar{k}}\varepsilon_{B}\right):{\rm det}H^{\bullet}\left(\Omega^*_{[0,a]}(M,E,H^{v})\right)
\longrightarrow{\rm
det}H^{\bullet}\left(\Omega^*_{[0,a]}(M,E,H^{w})\right)$$ is the
induced isomorphism on the determinant lines. Then we can compare it
with ${b_{{\rm
det}H^{\bullet}(\Omega^{\bullet}_{[0,a]}(M,E,H^{u}),d)}}$, similarly
as \cite[Lemma 3.7]{MW}, we get (\ref{4.11}).
\end{proof}

Combining Lemma \ref{t4.4} and Lemma \ref{t4.5}, we have
\begin{thm}
Let $M$ be a closed, compact manifold of odd dimension, $E$ be a
flat vector bundle over $M$. Suppose $H$ and $H'$ are closed
differential forms on $M$ of odd degrees representing the same de
Rham cohomology class, and let $B$ be an even form so that
$H'=H-dB$. Then the symmetric bilinear torsion $({\rm
det}\varepsilon_{B})^*\tau_{b,\nabla,H'}=\tau_{b,\nabla,H}$.
\end{thm}

\section{Compare with the refined analytic torsion}
\setcounter{equation}{0} In this section, we will compare the
symmetric bilinear torsion $\tau_{b,\nabla,H}$ with the refined
analytic torsion $\rho_{\rm an}(\nabla^{H})$ defined in \cite{H}.
The main theorem of this section is the following.
\begin{thm}\label{t5.1}
Let $M$ be a closed odd dimensional manifold, $E$ be a complex
vector bundle over $M$ with connection $\nabla$, $H$ be a closed
odd-degree differential form on $M$. Suppose there exists a
non-degenerate symmetric bilinear form on $E$. Then we have
\begin{align}\label{5.1}
\tau_{b,\nabla,H}\left(\rho_{\rm
an}\left(\nabla^{H}\right)\right)=\pm e^{-2\pi
i\left(\eta\left(\nabla^{H}\right)-{\rm rank}E\cdot\eta_{\rm
trivial}\right)},
\end{align}
where $\eta(\nabla^{H})$ and $\eta_{\rm trivial}$ are defined in
\cite{H}.
\end{thm}
We will use the method in \cite{BK4} to prove the theorem and the
proof will be given latter.

Let $h$ be a Hermitian metric on $E$. Then one can construct the
Ray-Singer analytic torsion as an inner product on ${\rm
det}H^{\bullet}(M,E,H)$, or equivalently as a metric on the
determinant line (cf. \cite[(6.13)]{H}). We denote the resulting
inner product by $\tau_{h,\nabla,H}$. Then by Theorem \ref{t5.1} and
\cite[Theorem 6.2]{H}, we get
\begin{cor}
If ${\rm dim}M$ is odd, then the following identity holds:
$$\left|{{\tau_{b,\nabla,H}}\over{\tau_{h,\nabla,H}}}\right|=1.$$
\end{cor}
\subsection{The dual connection}
Let $M$ be an odd dimensional closed manifold and $E$ be a flat
vector bundle over $M$, with flat connection $\nabla$. Assume that
there exists a non-degenerate symmetric bilinear form $b$ on $E$.
The dual connection $\nabla'$ to $\nabla$ on $E$ with respect to the
form $b$ is defined by the formula
$$db(u,v)=b(\nabla u,v)+b(u,\nabla' v),\ \ u,v\in\Gamma(M,E).$$
We denote by $E'$ the flat vector bundle $(E,\nabla')$.
\subsection{Choices of the metric and the spectral cut}
Till the end of this section we fix a Riemannian metric $g$ on $M$
and set
$\mathcal{B}^{H}=\mathcal{B}(\nabla^H,g)=\Gamma\nabla^{H}+\nabla^{H}\Gamma$
and
$\mathcal{B'}^{H}=\mathcal{B'}(\nabla'^H,g)=\Gamma\nabla'^{H}+\nabla'^{H}\Gamma$,
where $\Gamma:\Omega^{\bullet}(M,E)\to \Omega^{\bullet}(M,E)$ is the
chirality operator defined by
$$\Gamma \omega=i^{{n+1}\over 2}*(-1)^{{q(q+1)}\over 2}\omega,\ \ \omega\in\Omega^{q}(M,E).$$

We also fix $\theta\in(-\pi/2,0)$ such that both $\theta$ and
$\theta+\pi$ are Agmon angles for the odd signature operator
$\mathcal{B}^{H}$. One easily checks that
\begin{align}\label{5.2}
\left(\nabla^{H}\right)^{\#}=\Gamma \nabla'^{H}\Gamma,\ \
\left(\nabla'^{H}\right)^{\#}=\Gamma\nabla^{H}\Gamma,\ {\rm and}\
\left(\mathcal{B}^{H}\right)^{\#}=\mathcal{B}'^{H}.
\end{align}
As $\mathcal{B}^{H}$ and $(\mathcal{B}^{H})^{\#}$ have the same
spectrum it then follows that
\begin{align}
\eta\left(\mathcal{B}'^{H}\right)=\eta\left(\mathcal{B}^{H}\right)\
{\rm and}\ {\rm Det}_{{\rm
gr},\theta}\left(\mathcal{B}'^{H}\right)={\rm Det}_{{\rm
gr},\theta}\left(\mathcal{B}^{H}\right).
\end{align}
\subsection{A proof of Theorem \ref{t5.1}}
The symmetric bilinear form $\beta_{g,b}$ induces a non-degenerate
symmetric bilinear form
$$H^{j}(M,E')\otimes H^{n-j}(M,E)\longrightarrow \mathbb{C},\ \ j=0,\cdots,n,$$
and, hence, identifies $H^{j}(M,E')$ with the dual space of
$H^{n-j}(M,E)$. Using the construction of \cite[Section 5.1]{H}
(with $\tau:\mathbb{C}\to \mathbb{C}$ be the identity map) we obtain
a linear isomorphism
\begin{align}
\alpha:{\rm det}H^{\bullet}(M,E,H)\longrightarrow {\rm
det}H^{\bullet}(M,E',H).
\end{align}
\begin{lemma}\label{t5.2}
Let $E\longrightarrow M$ be a complex vector bundle over a closed
oriented odd-dimensional manifold $M$ endowed with a non-degenerate
bilinear form $b$ and let $\nabla$ be a flat connection on $E$. Let
$\nabla'$ denote the connection dual to $\nabla$ with respect to
$b$. Let $H$ be a closed odd-degree differential form on $M$. Then
\begin{align}\label{5.5}
\alpha\left(\rho_{\rm an}\left(\nabla^{H}\right)\right)=\rho_{\rm
an}\left(\nabla'^{H}\right).
\end{align}
\end{lemma}
The proof is the same as the proof of \cite[Theorem 5.3]{H}
(actually, it is simple, since $\mathcal{B}^{H}$ and
$\mathcal{B}'^{H}$ have the same spectrum and, hence, there is no
complex conjugation involved) and will be omitted.

For simplicity, we denote ${\rm
det}'(d_{\bar{0}}^{\#}d_{\bar{0}}|_{\Omega^{\bar{0}}_{(a,+\infty)}(M,E)})^{-1}\cdot({\rm
det}'(d_{\bar{1}}^{\#}d_{\bar{1}}|_{\Omega^{\bar{1}}_{(a,+\infty)}(M,E)}))$
by $\tau_{b,\nabla,H,(a,+\infty)}$. Let
$\Delta'^{H}=\left(\nabla'^{H}\right)^{\#}\nabla'^{H}+\nabla'^{H}\left(\nabla'^{H}\right)^{\#}$,
then we have
$$\Delta'^{H}=\Gamma\Delta^{H}\Gamma.$$
\begin{lemma}
The following identity holds,
\begin{align}\label{5.6}
\tau_{b,\nabla,H,(a,+\infty)}=\tau_{b,\nabla',H,(a,+\infty)}.
\end{align}
\end{lemma}
\begin{proof}
Applying (\ref{5.2}) and using the fact that
$$\nabla'^{H}:\Omega^{\bar{k}}_{(a,+\infty)}(M,E,H)\cap {\rm im}\left(\nabla'^{H}\right)^{\#}\to
\Omega^{\overline{k+1}}_{(a,+\infty)}(M,E,H)\cap {\rm
im}\nabla'^{H}$$ is an isomorphism, we get
\begin{multline}
\tau_{b,\nabla,H,(a,+\infty)}=\prod_{k=0,1}{\rm
det}'\left(\left(\nabla^{H}\right)^{\#}\nabla^{H}|_{\Omega^{\bar{k}}_{(a,+\infty)}(M,E,H)}\right)^{(-1)^{k+1}}\\
=\prod_{k=0,1}{\rm
det}'\left(\Gamma\nabla'^{H}\left(\nabla'^{H}\right)^{\#}\Gamma|_{\Omega^{\bar{k}}_{(a,+\infty)}(M,E,H)}\right)^{(-1)^{k+1}}\\
=\prod_{k=0,1}{\rm
det}'\left(\nabla'^{H}\left(\nabla'^{H}\right)^{\#}|_{\Omega^{\bar{k}}_{(a,+\infty)}(M,E,H)}\right)^{(-1)^{k}}\\
=\prod_{k=0,1}{\rm
det}'\left(\left(\nabla'^{H}\right)^{\#}\nabla'^{H}|_{\Omega^{\bar{k}}_{(a,+\infty)}(M,E,H)}\right)^{(-1)^{k+1}}=\tau_{b,\nabla',H}.\\
\end{multline}
\end{proof}

Then for any $h\in{\rm det}H^{\bullet}(M,E,H)$, we have
\begin{align}\label{5.8}
\tau_{b,\nabla,H}(h)=\tau_{b,\nabla',H}(\alpha(h)).
\end{align}
Then form (\ref{5.5}) and (\ref{5.8}) we get
\begin{align}\label{5.9}
\tau_{b,\nabla,H}\left(\rho_{\rm
an}\left(\nabla^{H}\right)\right)=\tau_{b,\nabla',H}\left(\rho_{\rm
an}\nabla'^{H}\right).
\end{align}

Let
$$\widetilde{\nabla}=\left(\begin{matrix}\nabla&0\\0&\nabla'\end{matrix}\right)$$
and
$$\widetilde{\nabla}^{H}=\left(\begin{matrix}\nabla^{H}&0\\0&\nabla'^{H}\end{matrix}\right).$$
Then for any $a\geq 0$, we have
$$\tau_{b,\widetilde{\nabla},H,(a,+\infty)}=\tau_{b,\nabla,H,(a,+\infty)}\cdot\tau_{b,\nabla',H,(a,+\infty)}$$
and
$$\tau_{b,\widetilde{\nabla},H}\left(\rho_{\rm an}\left(\widetilde{\nabla}^{H}\right)\right)
=\tau_{b,\nabla,H}\left(\rho_{\rm
an}\left(\nabla^{H}\right)\right)\cdot\tau_{b,\nabla',H}\left(\rho_{\rm
an}\left(\nabla'^{H}\right)\right).$$ Then combining the latter
equality with (\ref{5.9}), we get
$$\tau_{b,\widetilde{\nabla},H}\left(\rho_{\rm an}\left(\widetilde{\nabla}^{H}\right)\right)
=\tau_{b,\nabla,H}\left(\rho_{\rm
an}\left(\nabla^{H}\right)\right)^{2}.$$ Hence, (\ref{5.1}) is
equivalent to the equality
\begin{align}\label{5.10}
\tau_{b,\widetilde{\nabla},H}\left(\rho_{\rm
an}\left(\widetilde{\nabla}^{H}\right)\right) =e^{-4\pi
i\left(\eta\left(\nabla^{H}\right)-{\rm rank}E\cdot\eta_{\rm
trivial}\right)}.
\end{align}
By a slight modification of the deformation argument in
\cite[Section 4.7]{BK4}, where the untwisted case was treated, we
can obtain (\ref{5.10}). Hence, we finish the proof of Theorem
\ref{t5.1}.

\section{On the Cappell-Miller analytic torsion} \setcounter{equation}{0} In this section, we briefly
discuss the extension of the Cappell-Miller analytic torsion to the
twisted de Rham complexes. Let ${\rm dim}M$ be odd.

Using notations above, we have the twisted de Rham complex
$\nabla^{H}: \Omega^{\bar{k}}(M,E)\to \Omega^{\overline{k+1}}(M,E)$
and the chirality operator $\Gamma:\Omega^{\bar{k}}(M,E)\to
\Omega^{\overline{k+1}}(M,E)$, $k=0,1$. Define
$$d_{\bar{k}}^{\flat}=\Gamma d_{\bar{k}}\Gamma:\Omega^{\bar{k}}(M,E)\to \Omega^{\overline{k+1}}(M,E).$$
Then consider the non-self-adjoint Laplacian
$$\Delta_{\bar{k}}^{\flat}=\left(d_{\bar{k}}+d_{\bar{k}}^{\flat}\right)^{2}:\Omega^{\bar{k}}(M,E)\to \Omega^{\bar{k}}(M,E).$$
For any $a\geq 0$, let $\Omega^{\flat,\bar{k}}_{[0,a]}(M,E)$
($\Omega^{\flat,\bar{k}}_{(a,+\infty)}(M,E)$) denote the span in
$\Omega^{\bar{k}}(M,E)$ of the generalized eigensolutions of
$\Delta_{\bar{k}}^{\flat}$ with generalized eigenvalues with
absolute value in $[0,a]$ ($(a,+\infty)$). Then we have the
decomposition of the complex
$$\left(\Omega^{\bullet}(M,E),d\right)=\left(\Omega^{\flat,\bullet}_{[0,a]}(M,E),d\right)
\oplus\left(\Omega^{\flat,\bullet}_{(a,+\infty)}(M,E),d\right).$$
The subcomplex $(\Omega^{\flat,\bullet}_{[0,a]}(M,E),d)$ is a
$\mathbb{Z}_{2}$-graded finite dimensional complex. Then we can
define the torsion element $\rho^{\flat}_{\Gamma_{[0,a]}}\otimes
\rho^{\flat}_{\Gamma_{[0,a]}}\in{\rm
det}H^{\bullet}(\Omega^{\flat,\bullet}_{[0,a]}(M,E),d)^{2}\cong {\rm
det}H^{\bullet}(M,E,H)^{2}$, where $\rho^{\flat}_{\Gamma_{[0,a]}}$
defined by \cite[(2.22)]{H}. On the other hand, for the subcomplex
$(\Omega^{\flat,\bullet}_{(a,+\infty)}(M,E),d)$, the following
zeta-regularized determinant is well defined (cf. (\ref{3.4}))
\begin{align}
{\rm
det}'\left(d_{\bar{k}}^{\flat}d_{\bar{k}}|_{\Omega^{\bar{k}}_{(a,+\infty)}(M,E)}\right):=
\exp\left(-\zeta'\left(0,d_{\bar{k}}^{\flat}d_{\bar{k}}|_{{\rm
im}d_{\bar{k}}^{\flat}\cap\Omega^{\flat,\bar{k}}_{(a,+\infty)}(M,E)}\right)\right).
\end{align}

Consider the square of the graded determinant defined in
\cite[(2.38)]{H}, we find that for $\mathbb{Z}_{2}$-graded finite
dimensional complex $\Omega^{\flat,\bullet}_{(a,c)}(M,E)$, $0\leq
a<c<\infty$,
$${\rm det}'\left(d_{\bar{0}}^{\flat}d_{\bar{0}}|_{\Omega^{\flat,\bar{0}}_{(a,c)}(M,E)}\right)\cdot
{\rm
det}'\left(d_{\bar{1}}^{\flat}d_{\bar{1}}|_{\Omega^{\flat,\bar{1}}_{(a,c)}(M,E)}\right)^{-1}
=\left({\rm Det}_{\rm
gr}\left(\mathcal{B}_{\bar{0}}|_{\Omega^{\flat,\bullet}_{(a,c)}(M,E)}\right)\right)^{2}.$$
Then by \cite[Proposition 2.7]{H}, we easily get
\begin{prop}
The torsion element defined by
\begin{align}\label{6.2}
\rho^{\flat}_{\Gamma_{[0,a]}}\otimes
\rho^{\flat}_{\Gamma_{[0,a]}}\cdot \prod_{k=0,1}\left({\rm
det}'\left(d_{\bar{k}}^{\flat}d_{\bar{k}}|_{\Omega^{\bar{k}}_{(a,+\infty)}(M,E)}\right)\right)^{(-1)^{k}}\in{\rm
det}H^{\bullet}(M,E,H)^{2}
\end{align}
is independent of the choice of
$a\geq 0$.
\end{prop}
\begin{defn}
The torsion element in ${\rm det}H^{\bullet}(M,E,H)$ defined by
(\ref{6.2}) is called the twisted Cappell-Miller analytic torsion
for the twisted de Rham complex and is denoted by $\tau_{\nabla,H}$.
\end{defn}

Next we study the torsion $\tau_{\nabla,H}$ under metric and flux
deformations. Since the methods are the same as the cases in the
twisted refined analytic torsion \cite{H} and the twisted
Burghelea-Haller analytic torsion above, we only briefly outline the
results.

\begin{thm}{\rm (metric independence)}
Let $M$ be a closed odd dimensional manifold, $E$ be a complex
vector bundle over $M$ with flat connection $\nabla$ and $H$ be a
closed odd-degree differential form on $M$. Then the torsion
$\tau_{\nabla,H}$ is independent of the choice of the Riemannian
metric $g$.
\end{thm}
\begin{proof}
By the definition of $\tau_{\nabla,H}$ and the observation on the
determinants, this theorem follows easily from \cite[Proposition
2.4]{H}, \cite[(3.18)]{H} and \cite[(4.14)]{H}.
\end{proof}

\begin{thm}{\rm (flux representative independence)}
Let $M$ be a closed odd dimensional manifold, $E$ be a complex
vector bundle over $M$ with flat connection $\nabla$. Suppose $H$
and $H'$ are closed differential forms on $M$ of odd degrees
representing the same de Rham cohomology class, and let $B$ be an
even form so that $H'=H-dB$. Then we have $\tau_{\nabla,H'}={\rm
det}(\varepsilon_{B})\tau_{\nabla,H}$.
\end{thm}
\begin{proof}
From the above observation, this follows easily from \cite[Lemma
4.6]{H} and \cite[Lemma 4.7]{H}.
\end{proof}

\bibliographystyle{amsplain}

\end{document}